\documentclass[12pt]{amsart}
\usepackage{graphicx}
\usepackage{amsmath}
\usepackage{amstext}
\usepackage{amsfonts}
\usepackage{amsthm}
\usepackage{amssymb}

\DeclareMathOperator{\image}{im}

\newtheorem{theorem}{Theorem}
\newtheorem{definition}{Definition}
\newtheorem{example}{Example}
\newtheorem{lemma}{Lemma}
\newtheorem{corollary}{Corollary}

\newcommand{\Ar}{\mathbb{R}}
\newcommand{\Be}{\frak{B}}

\newcommand{\CSD}{\mathcal{L}}

\newcommand{\Rest}{\mathcal{R}}
\newcommand{\Mod}{\mathcal{M}}
\newcommand{\Pert}{\frak{P}}

\newcommand{\Energy}{\mathcal{E}}

\newcommand{\ep}{\epsilon}
\newcommand{\lam}{\lambda}

\newcommand{\simp}{\vec{t}}

\newcommand{\Bord}{\Omega F}

\newcommand{\HF}{HF}

\newcommand{\term}[1]{\textbf{#1}}

\newcommand{\tet}{\theta}

\newcommand{\GF}{\mathcal{F}}
\newcommand{\C}{\mathbb{C}}

\begin{document}
\title{Geometric Cycles in Floer Theory}
\author{Max Lipyanskiy}
\date{}

\address{Simons Center for Geometry and Physics, Stony Brook University, Stony Brook, NY 11794 }
\email{mlipyan@gmail.com}
\maketitle
     
\section{Introduction}
\subsection{Historical Overview}

In the mid 80's, Andreas Floer obtained a positive solution to Arnold's conjecture on the minimal number of fixed points of
a Hamiltonian symplectomorphism.  For this purpose, Floer introduced a new homology theory for the loop space of a symplectic
manifold.  His theory is an  infinite dimensional version of Morse theory applied to the symplectic action functional on
the loop space.  The critical points of this functional correspond to fixed points of the symplectomorphism, while the boundary
operator counts the dimension zero moduli spaces of (perturbed) holomorphic curves connecting
the critical points \cite{Floer1}.  Under suitable hypothesis, Floer was able to show that the resulting homology theory is well defined
and independent of the perturbation data necessary to construct the theory.  Moreover, Floer showed that the homology theory
was isomorphic to the singular homology of the underlying symplectic manifold and thus proved the the Arnold conjecture.  In subsequent work, Floer generalized his theory to other contexts such as the more general problem of Lagrangian intersections as well as an analogous theory for the Chern-Simons invariant for connections on a 3-manifold.  In these cases, a simple topological interpretation of the resulting groups is not available.  The groups encode deep geometric information about the relevant configuration space which cannot be reduced to the "classical" topology of that space.  \\\\
From a foundational standpoint, the definition of the Floer homology groups is perhaps not satisfactory.  The relevant functionals are usually not Morse-Smale and thus have to be perturbed in some manner to even define the groups.  As a result, one has to then show that the groups are indeed independent of the chosen perturbation.  Furthermore, a rather delicate analysis of the compactification of the moduli space of trajectories is necessary to establish even the most basic properties of the theory;  for instance, the fact that the chain of groups generated by critical points indeed form a complex.  The situation is of course analogous to the finite dimensional story.  One may take as the definition of homology of a compact manifold the chain complex associated to a Morse-Smale function.  However, establishing even the basic properties, such as functoriality under mappings, is quite nontrivial.  On the other hand, with singular homology theory at one's disposal,
Morse homology becomes an effective and illuminating way of computing the homology groups.  The central goal of the present work is to find an appropriate analogue of singular homology in the Floer context.     It should be emphasized, however, that while in the finite dimensional situation singular homology provides a way of avoiding the analytic machinery that is necessary for setting up Floer's theory, the theory developed in the current work rests heavily on the use of Sobolev spaces and appropriate nonlinear Fredholm operators between them.  This is, perhaps, a reflection of the fact that although many results in Floer theory have purely topological
interpretations, ultimately the theory deals with the qualitative behavior of solutions to certain elliptic partial differential equations.    \\\\
Let's briefly describe, what is to our knowledge,  the earliest evidence for the existence of such a theory.  In the late '80s, Atiyah \cite{Ati} and others, observed that, from the point of view of relative Donaldson invariants,  one may view Floer's theory as a theory of "semi-infinite cycles". We retell his observation in the language of Seiberg-Witten theory.  Consider a closed Riemannian 4-manifold $X$ with a spin$^\C$-structure and spinor bundle $W$ (see section $\ref{swcase}$ for the full definitions). Let $\Mod(X)$ be the moduli space of
solutions to the SW equations modulo the gauge group action.  As is well known (see for example \cite{KM}), when $b^+>0$
the moduli space is a  smooth compact manifold.  Now, consider the case when $X=X_+ \sqcup_Y X_-$ decomposes along a 3-manifold
$Y$ into two compact 4-manifolds with boundary.  Let $\Be(Y)$ denote the configuration space of pairs $(B,\Psi)$, where $B$ is a Clifford connection and $\Psi$ is a section of the spinor bundle over $Y$, modulo the gauge group action.  We have restriction maps $$\Rest_{\pm}:\Mod(X_{\pm})\rightarrow \Be(Y)$$  At least on the point-set level, one has $$\Mod(X)=\Mod(X_+)\times_{\Be(Y)}\Mod(X_-)$$  In other words, modulo the action of the gauge group, solutions on $X$ correspond to solutions on $X_{\pm}$ that agree on the boundary.  Standard elliptic boundary value theory implies that $\Mod(X_{\pm})$ are in fact Hilbert manifolds.  Therefore, one might hope to interpret the fibre product $\Mod(X_+)\times_{\Be(Y)}\Mod(X_-)$ in the smooth category.  Speculating even further, one might hope for the existence of Floer groups $\HF^+(Y)$ and $\HF^-(Y)$ with an intersection pairing $$\HF^+(Y) \otimes \HF^-(Y)\rightarrow H_*(\Be(Y))$$ where $H_*$ denote the singular homology functor. Abstracting this situation, given a smooth map $$\sigma:P\rightarrow \Be(Y)$$ where $P$ is some Hilbert manifold, we are led to the following problems:\\\\
1.  What properties should such maps have to have finite dimensional intersections?\\
2.  What properties should such maps have to have compact intersections?\\\\
The answer to the first problem is well known and involves the notion of a polarized Hilbert manifold.  Very loosely, one may think of a polarization of a Hilbert manifold such as $\Be(Y)$ as an equivalence class of local splittings of the tangent bundle: $$T\Be(Y)=T^+\Be(Y)\oplus T^-\Be(Y)$$  As it turns out, $\Be(Y)$ comes with a natural choice of polarization for which we have $$\pi^- \circ D\Rest_+:T\Mod(X_+)\rightarrow T^-\Be(Y)$$ Fredholm and $$ \pi^+ \circ D\Rest_+:T\Mod(X_+)\rightarrow T^+\Be(Y)$$ compact.  Similarly, $$\pi^+ \circ D\Rest_-:T\Mod(X_-)\rightarrow T^+\Be(Y)$$  is Fredholm while $$\pi^- \circ D\Rest_-:T\Mod(X_-)\rightarrow T^-\Be(Y)$$ is compact.  The fact that $\Mod(X_+)\times_{\Be(Y)}\Mod(X_-)$ is finite dimensional is an immediate consequence of Fredholm theory.  Therefore, to ensure finite dimensional intersections it is reasonable to require our cycles to respect the polarization.   The resolution of the second problem is significantly more subtle and is the main subject of this work.  As far as we know, the first attempt to do so is due to Tom Mrowka and Peter Ozsvath and the present dissertation owes a considerable debt to their original insight.\\\\
\textbf{Acknowledgement.}  We wish to thank Tom Mrowka for supervising the author's thesis which is the basis of this work.  In addition, we thank Peter Kronheimer and Dennis Sullivan for useful conversations.  Finally, we would like to thank the Simons Center For Geometry and Physics for their hospitality while this work was being completed.

\subsection{An Outline of the Contents}
Here we briefly describe the contents and provide some motivation for the constructions that follow. \\\\
As the basic structure we will be considering a Hilbert space $\Be$ with a polarization $T\Be \cong T^+\Be\oplus T^-\Be$ as well as a functional $\CSD:\Be \rightarrow \mathbb{R}$.  In section $\ref{MainC}$ we lay out the axioms for a map $\sigma:P\rightarrow \Be$ from a Hilbert manifold $P$ to define a cycle.  The motivation comes from the strong $L^2_1$ proof of compactness for Seiberg-Witten moduli spaces as presented in \cite{KM}.   On a more technical note, we will define and use the notion of locally cubical (lc) manifolds.  It appears to be a useful and technically simple structure to work with.    \\\\
In section $\ref{Loop}$ we will discuss our main example coming from Symplectic Geometry.  The example concerns the space of loops in $\C^n$.  We prove
the required $L^2_1$-compactness theorem. 
In section $\ref{Pert}$ we construct a family of perturbations for cycles ensuring that intersections can be arranged to be transverse.  The construction is quite similar to perturbing manifolds in finite dimensions and also rests on the Sard-Smale theorem.  One must simply check that the perturbation does not take us out of the category of cycles.  It is perhaps worth remarking that unlike in traditional Floer theory, where perturbations involve changing the metric, complex structure or the hamiltonian function, the perturbations here do not alter the geometric data.  Therefore, our theory is defines for a wide class of functionals that can potentially be highly degenerate.  This is illustrated in our proof of the existence of periodic orbits for loops in $\C^n$.  In this case the functional is degenerate and yet we are able to extract the relevant geometric information.  \\\\
In section $\ref{Corr}$ we define a general class of maps (called correspondences) $Z\rightarrow \Be_0\times \Be_1$ that give rise to maps on the Floer groups via fibre products.  The definition is general enough to include not only moduli spaces on a cobordism $W$ but also the diagonal map $\Be \rightarrow \Be \times \Be$.   The definition is a little technical but is forced on us by the considerations that follow.\\\\
Section $\ref{Triv}$ is the technical heart of the theory. Our goal is to prove that the trivial cobordism induces the identity on the Floer groups.  This is established by finding a cobordism between the correspondence coming from the moduli space of solutions on the cylinder and the diagonal map.  If $\Mod_t$ denotes the solutions on a cylinder of length $t$, we can form the disjoint union $\sqcup_{t\in (0,1]}\Mod_t$.  We complete this manifold by adding the diagonal $\Be \rightarrow \Be\times \Be$ at $t=0$.  This is achieved by arguing that on a small cylinder, a solution is specified by the appropriate spectral projections to the boundary.  For this we need to apply the contraction mapping theorem and in view of the nonlinear terms need rather precise estimates.   This establishes that we have a Hilbert manifold that with boundary $\Mod_1 \sqcup \Be$. However, the restriction maps defined on the cylinder extend only weakly  to the diagonal map as we approach the $t=0$ boundary.  Given a cycle $P\rightarrow \Be$ we show that by changing coordinates, we can assume that the difference map $\Rest_0-\sigma: \sqcup_{t\in (0,1]}\Mod_t\times P\rightarrow \Be$ is $C^1$ up to the boundary.  Such a coordinate change preserves the lc-structure but not the smooth structure on  $ \sqcup_{t\in (0,1]}\Mod_t\times P\rightarrow \Be$. This is the principal motivation for introducing lc-manifolds.  \\\\
In section $\ref{Appl}$ we illustrate the general theory by reproving that for a general class of hamiltonian functions $H:\C^n \rightarrow \Ar$, there exists a nontrivial periodic orbit.  Superficially, the proof is similar to the one in \cite{Hof}.  However, our proof is based on the unregularized gradient flow and does not use minimax methods.  It gives a rather natural interpretation of the cycles appearing in the construction.   We end with a couple of technical appendices. 
\section{Main Construction}
\label{MainC}

\subsection{lc-Manifolds of Depth $\leq 1$}
In this work it will be important to work with Hilbert manifolds with corners and
 some rather weak smoothness between different strata.  In an appendix, we will introduce a rather technical notion of locally-cubical manifolds or lc-manifolds for short.   For the sake of the reader, in this section we simply write down the definitions for the simplest nontrivial case.  In the terminology of lc-manifolds, this is a depth
  one lc-manifold. For the purposes of defining the bordism groups this is sufficient and illustrates all the essential technical difficulties.  Therefore, we decided to first give the definition in this special case.  While many of the propositions will be stated for general lc-manifolds, on first reading one may simply restrict to the case described below.
\begin{definition}
An lc-manifold of depth one is a Hausdorff space $P$,  with a distinguished closed subset called its boundary $\partial P \subset  P$.
We assume both $\partial P$ and $P-\partial P$ are Hilbert manifolds.  Furthermore,  each point $p\in \partial P$
has a neighborhood $U\subset \partial P$ and an open embedding $f:U\times [0,\ep)\rightarrow P$
such $f(u,0)=u$ while  $f_{|U\times \{ 0\}}$ and $f_{|U\times (0,\ep)}$ are diffeomorphisms.
\end{definition}
Let us call such a map $f:U\times [0,\ep)\rightarrow P$ an lc-chart.  Let $P$ be an lc-manifold of depth one and $\Be$ some Hilbert manifold.
\begin{definition}
 A continuous map $\sigma: P\rightarrow \Be$ is lc-smooth if the following hold:\\
1. $\sigma$ is smooth on $\partial P$ and $P-\partial P$ \\
2. Each point $p\in \partial P$ has an lc-chart $U\times [0,\ep)$ such that, in the chart coordinates,
$\sigma$ along with its first derivative in the $U$ direction is continuous on $U\times [0,\ep)$.
\end{definition}

Given an lc-smooth map $\sigma:P\rightarrow \Be$, we denote by $D\sigma_p:TP\rightarrow T\Be$ the differential restricted to the open stratum on which $p$ lives.
\subsection{Floer Spaces}
We will assume all our Hilbert manifolds to be separable.  Let $\Be$ be a Hilbert manifold.  For our examples, however, it suffices to consider the case of a Hilbert space. We have the following notion of polarization:
\begin{definition}
A polarization of $\Be$ is a direct sum decomposition $T\Be=T^+\Be \oplus T^-\Be$
\end{definition}
\begin{definition}
A \term{Floer space} ($\Be$,$\CSD$) is a polarized Hilbert manifold together with a continuous function  $\CSD: \Be \rightarrow \Ar$.  In addition, we assume $\Be$ is equipped with a coarser weak topology.
\end{definition}

\begin{definition}
A \term{chain} $\sigma: P\rightarrow \Be$ where $P$ is an lc-manifold  is an lc-smooth map satisfying the following axioms: \\ \\
\textbf{Axiom $1$.}  On $\image (\sigma)$, $\CSD$ is
bounded below and lower semi-continuous for the weak topology.  \\ \\
\textbf{Axiom $2$.}  Given a weakly converging sequence $\sigma(x_i)$ with limit $y$, if $\lim(\CSD(\sigma(x_i))=\CSD(y)$ then some subsequence of $x_i$ converges strongly on $P$.\\ \\
\textbf{Axiom $3$.}  Any subset $S\subset \image(\sigma)$ on which $\CSD$ is bounded is precompact for the weak topology. \\ \\
\textbf{Axiom $4$.}  $\Pi^-\circ D\sigma_p: TP \rightarrow T^-\Be$ is
Fredholm, $\Pi^+\circ D\sigma_p: TP \rightarrow T^+\Be$ is compact for each $p\in P$.
\end{definition}
\noindent \textbf{Remark.} A $\sigma$ satisfying Axiom 4 is said to be a \term{semi-infinite map}. \\\\
\textbf{Example.}  Take a Hilbert space $H=H^+\oplus H^-$ split into two infinite dimensional subspaces with its usual strong/weak topologies and $\CSD(v^+,v^-)=|v^-|^2-|v^+|^2$.  The polarization is given by the splitting.  $P=H^-$ with the inclusion map defines a cycle.
\begin{definition}
A chain $\sigma:P\rightarrow \Be$ has \term{index} $k$ if the linearized map $\Pi^-\circ D\sigma: TP \rightarrow T^-\Be$ has index $k$ at each point of $P-\partial P$.
\end{definition}
\textbf{Remark.} Note that $index(\sigma_{|\partial P})=index(\sigma)-1$.  Indeed,  since in an appropriate lc-chart around a point $p\in P$, $\sigma$ becomes $$\sigma:V\times [0,\ep)\rightarrow \Be$$ the differential of $\sigma$ in the $v$-variables is continuous on $V\times [0,\ep)$.  Therefore, the index of $D\sigma$ on $V\times (0,\ep)$ is exactly one greater than the index $D\sigma_{|V}$.
\begin{definition}
Two chains $\sigma_1:P_1\rightarrow \Be$,  $\sigma_2:P_2\rightarrow \Be$ are said to be isomorphic if there exists a diffeomorphism $f:P_1\rightarrow P_2$ such that $\sigma_2\circ f =\sigma_1$.
\end{definition}

\subsection{Floer Bordism}
The easiest invariant to define is a Floer Bordism Group:
\begin{definition}
A cycle is a chain  of depth 0.  In other words, $\partial P=\emptyset$
\end{definition}
\begin{definition}
 Let $\Bord_k(\Be,\CSD)$ be the $\mathbb{Z}_2$-vector space generated by isomorphism classes of cycles of index $k$.  Disjoint union is the additive structure.  Furthermore, $[P]=0$ if $\sigma:P\rightarrow \Be$ extends to a chain of depth one $\sigma':W\rightarrow \Be$ with $\partial W=P$ and $\sigma'_{|\partial W}=\sigma$. Let $\Bord_*(\Be,\CSD)=\bigoplus_k\Bord_k(\Be,\CSD)$.
\end{definition}

\begin{definition}
Given $(\Be,\CSD)$ as above, let $-\Be$ be the polarized Hilbert space obtained by switching  $T^{+}\Be$  and $T^-\Be$ and let $(-\Be,-\CSD)$ be the Floer space obtained by switching the sign of $\CSD$.
\end{definition}
The motivation for our definition of chain is the following result:
\begin{lemma}
Given cycles $\sigma:P\rightarrow \Be$ and $\tau: Q\rightarrow -\Be$,  their intersection $\sigma \cap \tau =P\times_{\Be} Q$  is compact.
\end{lemma}
\begin{proof}
 On the image of the intersection, $\CSD$ is bounded above and below.  Therefore, Axiom 3 implies that this image is weakly precompact.  Furthermore, since on the image $\CSD$ is both lower and upper semi-continuous, it is continuous in the weak topology.   Axiom 2 implies every sequence $x_i \in \sigma \times_\Be \tau $ must have a convergent subsequence.
\end{proof}

Let us make the following perturbation hypothesis which will be verified for example we consider. \\\\
\textbf{Existence of Perturbations:}  Given cycles $\sigma:P\rightarrow \Be$ and $\tau: Q\rightarrow -\Be$ there exists a chain $F:P\times [0,1]\rightarrow \Be$ with
$F_{|P\times 0}=\sigma$ and $F_{|P\times 1}$ transverse to $\tau$.  Furthermore, if $\sigma$ is already transverse to $\tau$, without changing $F_{|P\times 1}$, we may alter $F$ to be transverse to $\tau$ as well.

\begin{theorem}Given cycles $\sigma$ and $\tau$ as above, having transverse intersection, their fibre product $\sigma \times_\Be \tau $ is a closed manifold mapping to $\Be$.  The fibre product gives a well-defined map
$$ \Bord_k(\Be,\CSD)\times \Bord_l(-\Be,-\CSD)\rightarrow \Omega_{k+l}(\Be)$$ where $\Omega_{k+l}(\Be)$ denotes ordinary lc-bordism with $\mathbb{Z}_2$-coefficients.
\end{theorem}
\begin{proof} Note that we can view $\sigma \times_\Be \tau $  as $(\sigma\times \tau)^{-1}(\Delta)$ where $\Delta$ is the diagonal in $\Be\times \Be$.  By assumption, $\sigma\times \tau$ is transverse to $\Delta$  and thus $\sigma \times_\Be \tau $ is a smooth finite dimensional lc-manifold.  To calculate the dimension, note that locally $\sigma \times_\Be \tau $ is $(\sigma-\tau)^{-1}(0)$.  Up to compact perturbation, the linearized operator has the form $$\begin{pmatrix} \Pi^-\circ D\sigma &0 \\ 0 & -\Pi^+ \circ D\tau \end{pmatrix}$$  Thus, the dimension of $\sigma \times_\Be \tau $ is $ind(\Pi^-\circ D\sigma)$+$ind(\Pi^+\circ D\tau)$.  Finally, if $F:W\rightarrow \Be$ is an chain with $\partial F=\sigma$, we have $$\partial(F \times_{\Be} \tau)=\partial F \times_{\Be} \tau =\sigma \times_{\Be} \tau $$ when $F$ is transverse to $\tau$.  Therefore, the fibre product descends to a map on $\Bord_*(\Be,\CSD)$.
\end{proof}
\textbf{Remark.}  In the context of Floer theory discussed in this work one can modify the definition of $\Bord_*(\Be,\CSD)$ so that the fibred product lies in the usual smooth bordism groups, rather than the lc-bordism groups.  However, since our primary interests is in homology rather than bordism we do not develop this here.

\section{Loop Space of $\C^n$}
\label{Loop}

\subsection{Semi-Infinite Cycles for the Action Functional on ${L^2_{1/2}(S^1,\C^{n}})$}
\label{CycleDef}
We now turn to our main example. Let $\Be$ be the Hilbert space of $L^2_{1/2}$ loops on $\C^{n}$.  Explicitly, if a loop $\gamma$ is decomposed in Fourier series $$\gamma(\theta)=\sum_n c_n e^{in\theta}$$ then the square of the $L^2_{1/2}$-norm of $\gamma$ is $$\sum_n|c_n|^2|n|+|c_0|^2$$
 Given a smooth  function $H:\C^{n}\rightarrow \Ar$ we define the action functional by:
$$\CSD_H(\gamma)=\int\frac{1}{2}\langle -J\dot{\gamma}(\tet),\gamma\rangle -H(\gamma(\tet))d\tet$$
where $J$ is the standard complex structure in $\C^{n}$.
The formal $L^2$-gradient of $\CSD_H$ is:$$\nabla{\CSD_H}(\gamma)=-J\partial_\theta \gamma-\nabla H(\gamma)$$  %Therefore, the upward gradient flow of $\nabla{\CSD_H}$ is: $$\partial_t u(t,\tet)=\nabla{\CSD_H}(u(t,\tet))=-J\partial_\theta u(t,\tet)-\nabla H(u(t,\tet))$$  We write this as a perturbed $J$-holomorphic curve equation as:$$\partial_t u(t,\tet)+J(\partial_\theta u(t,\tet)-J\circ \nabla H(u(t,\tet)))=0$$  Therefore, for a $u:[0,T]\times S^1\rightarrow \C^{n}$ satisfying the perturbed holomorphic curve equation we have: $$\CSD_H(u(T,\cdot))-\CSD_H(u(0,\cdot))=\int_0^T||u(t,\cdot)||_{L^2}dt=E(u)$$

Since $\Be$ is a linear space, we may define the polarization by the splitting $T\Be=T^+\Be\oplus T^-\Be$ where $T^+\Be$ is spanned by the positive eigenvectors of $-J\partial_\tet$ and $T^-\Be$ by the nonpositive eigenvectors of $-J\partial_\tet$.  We have constructed a Floer space and thus have an associated Floer  group $\Omega F_*(\Be,\CSD_H)$.\\\\
From now on, assume  $H$ is smooth with $H=0$ near $0$ and $H(x)=(1+\ep)|x|^2$ for $|x|$ large.   Fix a unit vector $e^+\in T^+\Be$.  We construct cycles for this Floer space.  Following \cite{Hof}, we have distinguished subsets: $$\Sigma_\tau=\{\gamma|\gamma^-+se^+,||\gamma^-||_{L^2_{1/2}}\leq \tau, 0\leq s \leq \tau\}$$
and
$$\Gamma_\alpha=\{ \gamma\in T\Be^+,||\gamma||_{L^2_{1/2}} =\alpha\} $$
It is elementary to show (see \cite{Hof}) that for $\tau \gg 1$, ${\CSD_H}_{|\partial \Sigma_\tau}\leq 0$ and there exists $\alpha>0$ and $\beta>0$ such that ${\CSD_H}_{|\Gamma_\alpha}\geq \beta$.
Note that $\Sigma_\tau \cap \Gamma_\alpha=\{\alpha e^+\}$ transversely.
\begin{lemma}
$\Sigma_\tau$ is a cycle for $(-\Be,-\CSD_H)$ and $\Gamma_\alpha$ is a cycle for $(\Be,\CSD_H)$.
\end{lemma}
\begin{proof}
The proofs are nearly identical so let us focus on $\Gamma_\alpha$.  The key observation is that $$\CSD_H(\gamma^+)=\int\frac{1}{2}\langle -J\dot{\gamma}^+(\tet),\gamma^+\rangle -H(\gamma^+(\tet))d\tet=\frac{1}{2}||\gamma^+||_{L^2_{1/2}}^2-\int H(\gamma^+(\tet))d\tet$$  From this it follows that the action functional on $\gamma^+$ essentially coincides with the $L^2_{1/2}$ norm.  Indeed, we may write $H(x)=H_c(x)+(1+\ep)|x|^2$ where $H_c$ has compact support.  Given $\gamma^+\in \Gamma_\tau$ we have $||\gamma^+||_{L^2_{1/2}}$ bounded uniformly.  Therefore, $\CSD_H$ is bounded.  Lower semicontinuity follows from the fact that the $L^2_{1/2}$ norm can only drop in a weak limit and the fact  that $H_c$ is continuous for the weak topology.  Given that $\CSD_H$ does not drop implies the $L^2_{1/2}$ does not drop which in turn implies convergence.
\end{proof}

\section{Perturbations}
\label{Pert}
To ensure transverse intersection of cycles we need to be able to perturb them with a sufficiently large parameter
space at the same time ensuring that the perturbed map is still a  chain. We show how to construct such perturbations for the loop space.
\begin{definition}
Let $\Pert \subset \Be= L^2_{1/2}(S^1;\C^n)$ be the unit ball in the $L^2_2$-norm. From the compactness of the inclusion $L^2_5\subset C^1$ we have that every sequence $v_i\in \Pert$ has a $C^1$ convergent subsequence.
\end{definition}
 Let $\rho$ be a positive bump function equal to 1 on $[-1,1]$ and  $1/x^2$ outside $[-2,2]$.  We define the map   $$F:\sigma\times \Pert \rightarrow \Be$$ by
$$F(x,v)=\sigma(x)+\rho(||\sigma(x)||^2_{L^2_{1/2}})v$$
We have the following theorem:
\begin{theorem}
Given a chain $\sigma:P \rightarrow \Be$ the map $F:P \times \Pert \rightarrow \Be(Y)$ satisfies:\\ 
1. $F(x,0)=\sigma(x)$ \\ 2. $DF_{(x,v)}$ has dense image for all $(x,v)$ \\ 3. $\Pi^- \circ D_xF$ is Fredholm, $\Pi^+\circ D_xF$ and $D_v F$ are compact  \\
4. Given a compact lc-manifold $K \subset \Pert$, $F_{|P\times K}$ is again a
semi-infinite chain.
\end{theorem}
\begin{proof}  
\textbf{Claim:} $F$ satisfies all the requirements of the theorem.  \\ \\
Part 1: Clear from the construction. \\\\
Part 2: Note that  $$DF_{(x,v)}(0,w)=\rho(||\sigma(x)||^2_{L^2_{1/2}})w$$
and since $L^2_2(S^1;\C^n)\subset \Be$ is  compact with dense image and $\rho(||\sigma(x)||^2_{L^2_{1/2}})>0$ we have that $DF$ has dense image and $D_vF$ is compact.\\\\
Part 3:  $$DF_{(x,v)}(y,0)=D\sigma_x(y)+T(y)v$$ where $T:TP\rightarrow \mathbb{R}$ is the linear map given by $T(y)=D\rho(y)$  Therefore,viewed as a map on $TP$, $DF(y,0)$ differs from $D\sigma$ by an at most rank one map from which part 3 follows. \\ \\
Part 4: Observe that $\CSD(F(x,v))-\CSD(\sigma(x))$ is bounded independent of $x$.   For this, write $$\CSD(\gamma)=Q(\gamma, \gamma)+G(\gamma)$$ where $Q$ is a bilinear form and $G$ is a bounded function.  Let $a=\sigma(x)$ and $b=v\rho(||a||^2_{L^2_{1/2}})$.  To bound  $\CSD(F(x,v))-\CSD(\sigma(x))$ we need to bound $Q(a,b)$ and $Q(b,b)$.  This follows from the definition of the perturbation.

   Also, observe that if $\sigma(x_i)$ are weakly $L^2_{1/2}$ convergent   $\CSD(\tilde{F}(x_i,v_i))-\CSD(\sigma(x_i))$ is in fact strongly convergent, after passing to a subsequence.  This follows from the fact that $b_i$  is $C^1$ precompact.  This implies $\CSD(F(x_i,v_i))$ drops exactly when $\CSD(\sigma(x_i))$ drops. Let us check that $P\times K\rightarrow \Be$ satisfies all the axioms.  Given a weakly convergent sequence $\sigma(x_i)+c_i\cdot v_i$ note that $c_i\cdot v_i$ converge strongly  since $K$ is $L^2_{1/2}$ precompact.  Therefore, $\sigma(x_i)$ is weakly convergent and lower semi-continuity follows.  If $\CSD$ does not drop in the limit, we must have $x_i$ precompact and thus $(x_i,v_i)$ precompact as well.\end{proof}

%The following definition will be important for our discussion of homology.  Here we investigate how it behaves under perturbations.

%\begin{definition}
%A set $S\subset \Be$ is \term{$k$-negligible} if it is contained in the image of a semi-infinite Hilbert manifold $Q\rightarrow \Be$ of index $<k$.
%\end{definition}

%\begin{definition}
%A chain $\sigma:P\rightarrow \Be$ of index $k$ is said to be \term{negligible} if its image is $k$-negligible.
%\end{definition}
%\begin{theorem}

 %If $\sigma:P\rightarrow \Be$ is $k$-negligible then $F(\cdot,v):P\rightarrow \Be$ is $k$-negligible.

%\end{theorem}
%\begin{proof}
 %If given $\sigma(p)=g(q)$ where $g:Q\rightarrow \Be$ semi-infinite of smaller index then $F(p,v)=\sigma(p)+\rho(||\sigma(p)-(B_0,0)||^2_{L^2_{1/2}})v=g(q)+\rho(||g(q)-(B_0,0)||^2_{L^2_{1/2}})v$ \end{proof}
We can now use the perturbations to put cycles in general position:
\begin{theorem}
Given a cycle $\sigma:P \rightarrow (\Be,\CSD)$ and a cycle $\tau:Q\rightarrow (-\Be,-\CSD)$ there exists a cobordant cycle $\sigma':P\rightarrow (\Be,\CSD)$ such that $\sigma'$ is transverse to  $\tau$.  Furthermore, given two such transverse cycles $\sigma'$ and $\sigma''$ there exists a cobordism $\Sigma:P\times [0,1]\rightarrow (\Be,\CSD)$ transverse to $\tau$, with $\partial \Sigma=\sigma'-\sigma''$.
\end{theorem}
\begin{proof}
 The argument follows the standard route.  The map $F \times \tau \rightarrow \Be\times \Be$ is transverse to $\Delta\subset \Be\times \Be$.  The map $(F\times \tau)^{-1}(\Delta)\rightarrow \Pert$ is Fredholm.  Hence, applying Smale's extension of Sard's theorem \cite{Smale}, we have that for generic $p\in \Pert$ the map $F(,p)\times \tau$ is transverse to $\Delta$.  Any to such generic values $p,q$ may be connected by an arc $\gamma:[0,1]\rightarrow \Pert$ transverse to $(F\times \tau)^{-1}(\Delta)\rightarrow \Pert$ \end{proof}
%\textbf{Remark.}  For some purposes it is necessary to make the perturbations in some small $L^2_{1/2}$ neighborhood of a point.  In this case, one would take $\rho$ to be supported in that ball.  The rest of the proof remains the same.

\section{Correspondences}
\label{Corr}
\subsection{Definitions}
We now explain how to obtain maps between bordism groups of Floer spaces:

\begin{definition}
A correspondence $(Z,f)\in Cor((\Be_0,\CSD_0),(\Be_1,\CSD_1))$ is a map $f:Z\rightarrow \Be_0 \times \Be_1$ where $Z$ is a Hilbert
manifold (possibly with boundary) satisfying the following axioms: \\ \\
\textbf{Axiom $1'$.}  On $\image(f)$, $\CSD_1-\CSD_0$ is bounded below and lower semi-continuous for the weak topology.  \\ \\
\textbf{Axiom $2'$.} If $\CSD_1(\pi_1(z_i))$ is bounded above and $\pi_0(z_i)$ is a weakly
 precompact sequence then $\pi_1(z_i)$  weakly precompact. \\ \\
\textbf{Axiom $3'$.} Given $\pi_1(z_i)$ is weakly convergent to $x$,  if $\lim f^*(\CSD_1-\CSD_0)(z_i)=(\CSD_1-\CSD_0)(x)$
and $\pi_0(z_i)$ is converges strongly  then $z_i$ converges strongly (up to a subsequence). \\ \\
\textbf{Axiom $4'$.} $(\pi_0^+,\pi^-_1)\circ Df: TZ\rightarrow
T^+\Be_0\oplus T^-\Be_1$ is Fredholm.  Given a bounded sequence $v_i\in TZ$, if $\pi_0^+(Df)(v_i)$ is weakly convergent,  $\pi_1^+(Df)(v_i)$ is precompact.\\ \\
\textbf{Axiom $5'$.} $Df:TZ\rightarrow T\Be_0$ is dense. $Df_{|\partial(Z)}:TZ\rightarrow T\Be_0$ is also dense.
\end{definition}
\begin{example}  The diagonal map $\Delta: \Be\rightarrow \Be\times\Be$ is a correspondence.
\end{example}
\begin{theorem}
Given a chain $\sigma: P \rightarrow \Be_0$ and a correspondence $f:Z\rightarrow \Be_0 \times \Be_1$, the
fiber product $\pi_1\circ f:P \times_{\Be_0}Z\rightarrow \Be_1$ is chain in $(\Be_1,\CSD_1)$.
\end{theorem}
\begin{proof}
 Axiom 1: $-\CSD_0+\CSD_1>C$ and $\CSD_0>C$ imply $\CSD_1>2C$.  Given a sequence $(x_i,z_i)\in P \times_{\Be(Y_1)}Z$
with a weakly convergent sequence $f(z_i)$, we have $$\liminf(-\CSD_0(f(z_i)))+\CSD_1(f(z_i)) \geq
-\CSD_0(f(z_\infty))+\CSD_1(f(z_\infty))$$ and  $$\liminf \CSD_0(f(z_i)) \geq \CSD_0(f(z_\infty))$$ imply
$\liminf \CSD_1(f(z_i)) \geq \CSD_1(f(z_\infty))$.  \\ \\
Axiom 2: If  $\lim\CSD_1(f(z_i)) = \CSD_1(f(z_\infty))$ Axiom $1'$ implies that $\CSD_0$ can only rise in the limit.  However,
since $\sigma(x_i)$ are assumed weakly convergent as well, we have $\lim\CSD_0(\sigma(x_i)) = \CSD_0(\sigma(x_\infty))$ and thus
$\sigma(x_i)=\pi_0(f(z_i))$ is strongly convergent so Axiom $3'$ implies $z_i$ strongly convergent as well. \\ \\
Axiom 3: $\CSD_1(f(z_i))<C$ implies $\CSD_0(f(z_i))< C$ and thus
$\sigma(x_i)$ is weakly precompact.  Axiom $2'$ implies that $\pi_1(z_i)$ is weakly precompact as well.  \\ \\
Axiom 4: We use the following lemma
\begin{lemma}
Given a linear Fredholm map $T:W\rightarrow V_1\oplus V_2$ such that $\Pi_1\circ T$ is surjective, $T_{|\ker(\Pi_1\circ T)}\rightarrow V_2$ is Fredholm with the same index.
\end{lemma}
\begin{proof} We have $\ker(T)=\ker(T_{|\ker(\Pi_1\circ T)})$.  Surjectivity of $\Pi_1\circ T$ implies  the dimension of cokernel coincides as well.
\end{proof}
Take a polarization of $T\Be_0\oplus T\Be_1$ with projections $(\Pi_i^+,\Pi_i^-$).  We apply this lemma to the map $$F:TZ\times TP \rightarrow T^+\Be_0\oplus T^-\Be_0 \oplus T^-\Be_1$$ with $$F(z,p)=(\Pi^+_0(Df(z))-\Pi^+_0(D\sigma(v)))\oplus (\Pi^-_0(Df(z))-\Pi^-_0(D\sigma(p)))\oplus \Pi^-_1(Df(z))$$ Axiom $6'$ implies the above lemma applies since $\image Df\oplus D\sigma$ in $T\Be_0$ is closed and dense.  To calculate the index deform through Fredholm operators to $$\tilde{F}(z,p)=\Pi^+_0(Df(z))\oplus \Pi^-_0(D\sigma(p))\oplus \Pi^-_1(Df(z))$$ Thus, with respect to the given polarization, $$\dim(Z\times_{\Be(Y_0)}P)=ind(Df)+ind(D\sigma)$$
\end{proof}

\begin{lemma}
A correspondence $F:Z \rightarrow \Be_0\times \Be_1$ of index $k$ without boundary induces a map:
% $$\HF_k(F):\HF_*(\Be_0,\CSD_0)\rightarrow \HF_{*+k}(\Be_1,\CSD_1)$$
$$\Bord_k(F):\Bord_*(\Be_0,\CSD_0)\rightarrow \Bord_{*+k}(\Be_1,\CSD_1)$$
\end{lemma}
\begin{proof}
Given a cycle $\sigma:P\rightarrow \Be_0$ we have $\partial(\sigma\times_{\Be_0} F)=\partial(\sigma)\times_{\Be_0}F$ since $\partial F=\emptyset$  This shows that $\Bord_k(F)$ commutes with the boundary operator.
%By the index calculation above, $\HF_k(F)$ takes $m$-negligible sets to $(m+k)$-negligible ones.
\end{proof}
%Similarly, we have a version of the homotopy axiom:
%\begin{theorem}
%Given a correspondence $F:Z \rightarrow \Be_0\times \Be_1$ with boundary $(F_0\sqcup F_1,Z_0\sqcup Z_1)$ we have:
%$$\Bord(F_0)=\Bord(F_1)$$
%\end{theorem}
%\begin{proof} Given a cycle $\sigma:P\rightarrow \Be$, $$\partial(\sigma\times_{\Be} F)=\sigma \times_{\Be} \partial(F)=\sigma \times_{\Be}(F_1 \sqcup F_0)$$
%\end{proof}

\section{The Trivial Cobordism}
\label{Triv}
\subsection{$L^2_1$ Compactness for a Holomorphic Cylinder}
In this section we explain how a holomorphic cylinder gives rise to a correspondence.  In fact, we will demonstrate that this correspondence induces identity on Floer bordism.    The intuition behind the proof is the observation that a trivial cobordism is the analogue of a gradient flow in finite dimensions and letting the flow time shrink to zero induces the identity map on the underlying manifold.  In our infinite dimensional setting the initial value problem is not well-defined and thus our homological argument is meant as a substitute notion.\\\\
Take $S^1$ to be the standard circle of length $2\pi$.  Let $Z_T=[0,T]\times S^1$ be the cylinder with coordinates $(t,\theta)$ and complex structure $j(\partial_{t})=\partial_\theta$.   Let $H:\C^{n}\times S^1 \rightarrow \Ar$ be a Hamiltonian with associated vector field $X_H=J\circ \nabla H$.  Given an $L^2_1$-map $u:Z_T\rightarrow \Ar^{2n}$ we define the energy to be $$E(u)=\frac{1}{2} \int_0^T \int_0^{2\pi} |u_t|^2+|u_\theta-X_H(u,t)|^2d\theta ds$$
The upward gradient flow of $\nabla{\CSD_H}$ is: $$\partial_t u(t,\tet)=\nabla{\CSD_H}(u(t,\tet))=-J\partial_\theta u(t,\tet)-\nabla H(u(t,\tet))$$  We write this as a perturbed $J$-holomorphic curve equation as:
\begin{equation}
\label{JCurv}
\partial_t u(t,\tet)+J(\partial_\theta u(t,\tet)-J\circ \nabla H(u(t,\tet)))=0
\end{equation}
  Therefore, for a $u:[0,T]\times S^1\rightarrow \C^{n}$ satisfying the perturbed holomorphic curve equation we have: $$\CSD_H(u(T,\cdot))-\CSD_H(u(0,\cdot))=E(u)$$

Consider $X_H(u,t)=c\cdot u+X_{H_c}(u,t)$ where $$X_{H_c}(u,t):\Ar^{2n}\times S^1 \rightarrow \Ar^{2n}$$ is $C^1$ and compactly
supported and $c\in \sqrt{-1}\Ar$.
\begin{lemma}
 $X_{H_c}(u,t)$ is continuous in $u$ for the $L^2$ topology.
\end{lemma}
\begin{proof}
 Since $X_{H_c}(u,t)$ has compact support, we have $$|X_{H_c}(v,t)-X_{H_c}(v',t)|\leq C|v-v'|$$  Therefore, $$|X_{H_c}(v,t)-X_{H_c}(v',t)|^2\leq C^2|v-v'|^2$$
Given $u_1,u_2:S^1\rightarrow \Ar^{2n}$, we have $$\int|X_{H_c}(u_1,t)-X_{H_c}(u_2,t)|^2dt\leq C^2\int|u_1-u_2|^2dt=C^2||u_1-u_2||_{L^2}^2$$
\end{proof}

\begin{theorem}
\label{compact}
Assume $c\notin \sqrt{-1}\mathbb{Z}$.  Given  a sequence with $E(u_i)<C$, the $u_i$  are uniformly $L^2_1$ bounded and thus weakly precompact.  Given a weakly convergent
sequence $u_i$, we have $E(u_\infty)\leq \liminf E(u_i)$.  If $\lim E(u_i)=E(u_\infty)$,
the $u_i$ converge strongly in $L^2_1$ to $u_\infty$. Finally, if $c\in \sqrt{-1}\mathbb{Z}$,
the theorem applies if we furthermore assume  $||u_i(0,\theta)||_{L^2} < C'$.
\end{theorem}
\begin{proof} We first prove the theorem when $X_{H_c}=0$. When $c\notin \sqrt{-1} \mathbb{Z}$
 we have an isomorphism $$\partial_\theta+c:L^2_1(S^1)\rightarrow L^2(S^1)$$ Thus,
 $$\frac{1}{2} \int_0^T \int_0^{2\pi} |u_t|^2+|u_\theta-X_H(u,t)|^2d\theta ds=const\cdot ||u||_{L^2_1}^2$$
Thus, the energy is equivalent to the $L^2_1$ norm from which everything follows.
When $c \in \sqrt{-1} \mathbb{Z}$ a special argument is needed.
Since we have $$\int_0^T \int_0^{2\pi} |u_t|^2d\theta dt < C$$ we get
 $$||u(0,\theta)|_{L^2}-|u(\tau,\theta)|_{L^2}|\leq \tau^{1/2}||u_t||_{L^2}$$ for all $\tau \in [0,T]$.
 This, together with the $L^2$ bound on $u(0,\theta)$, implies an $L^2$ bound on $u$.
  The bound on $u_\theta$ follows since  we have bounds on $u_t$ and $X_H(u)=c\cdot u$.
  Now, assume $u_i$ are weakly convergent.  We may rewrite the energy as $$\frac{1}{2} \int_0^T \int_0^{2\pi} |u_{i,t}|^2+|u_{i,\theta}|^2+|c\cdot u_i|^2+2\langle u_{i,\theta},c\cdot u_i \rangle d\theta ds$$  Weak $L^2_1$ convergence of $u_i$ implies $\int_0^T \int_0^{2\pi} \langle u_{i,\theta},c\cdot u_i \rangle d\theta ds$ converges to $$\int_0^T \int_0^{2\pi} \langle u_{\infty,\theta},c\cdot u_\infty \rangle d\theta ds$$ and thus lower semicontinuity follows from that of the $L^2_1$ norm.
The case $X_{H_c}\neq 0$ is a slight modification of the argument.
We observe that $$|(u_\theta-c\cdot u)-X_{H_c}(u,t)|^2= |u_\theta-c\cdot u|^2+|X_{H_c}(u,t)|^2+2\langle u_\theta-c\cdot u, X_{H_c}(u,t) \rangle$$
Hence $$|(u_\theta-c\cdot u)-X_{H_c}(u,t)|^2 \geq |u_\theta-c\cdot u|^2+|X_{H_c}(u,t)|^2-2|\langle u_\theta-c\cdot u,X_{H_c}(u,t) \rangle|\geq$$ $$|u_\theta-c\cdot u|^2+|X_{H_c}(u,t)|^2-|u_\theta-c\cdot u|^2/2-2|X_{H_c}(u,t)|^2=1/2|u_\theta-c\cdot u|^2-|X_{H_c}(u,t)|^2 $$  Since we have an $L^\infty$ bound on $X_{H_c}(u,t)$, a bound on $E(u)$ is the same as a bound on $$\frac{1}{2} \int_0^T \int_0^{2\pi} |u_t|^2+|u_\theta-c\cdot u|^2d\theta ds$$  In addition, since $X_{H_c}(u,t)$ is continuous for the $L^2$ topology, the $\lim_{i\rightarrow \infty}E(u_i) =E(u_\infty)$  exactly when $$\lim_{i\rightarrow \infty}\frac{1}{2} \int_0^T \int_0^{2\pi} |u_{i,t}|^2+|u_{i,\theta}-c\cdot u_i|^2d\theta ds =\frac{1}{2} \int_0^T \int_0^{2\pi} |u_{\infty,t}|^2+|u_{\infty,\theta}-c\cdot u_\infty|^2d\theta ds$$

\end{proof}
\textbf{Remark.}  In the case $c\in \sqrt{-1}\mathbb{Z}$ the assumption on $u$ may seem artificial.  However, this assumption
is exactly met when describing the axioms of a correspondence.  Therefore, the holomorphic maps on a cylinder will always
give rise to a correspondence under the above assumptions.\\\\
Let $M_T$ be the moduli space of solutions to equation $\ref{JCurv}$ on a cylinder of length $T$.   Restriction maps to $0\times S^1$ and $T\times S^1$ induce maps $\Rest_0:M_T\rightarrow \Be$, $\Rest_T:M_T\rightarrow \Be$.

\begin{lemma}
Given $T>0$, $Z_T$ together with the restriction maps define a correspondence $$\Rest_0\times \Rest_T:M_T\rightarrow \Be\times -\Be$$
\end{lemma}
\begin{proof}
This follows from the compactness theorem above together with the unique continuation property.  
\end{proof}

\subsection{APS Boundary Value Problem }
Consider $D=\partial_t+L$ where $L=J\partial_\theta$ is a first order elliptic differential operator acting on $\C^n$ valued maps on $S^1$.  We have the APS boundary value problem: $$(D_\ep,\Pi^-_{L}\circ
r_\ep-\Pi^+_{L}\circ r_0):L^2_1([0,\ep]\times S^1,\C^n)\rightarrow L^2([0,\ep]\times Y,\C^n)\oplus L^2_{1/2}(Y,\C^n)$$
Where $\Pi^+$ is the spectral projection to the nonnegative part of the spectrum of $L$ and $\Pi^-$ is the projection to the negative part.\\\\
Thus, given $\alpha \in L^2_1([0,\ep]\times S^1,\C^n)$ the boundary data is specified by $\beta=\beta^+_0+\beta^-_\ep$ with
$\beta^+_0=-\Pi^+_{L}\circ r_0 (\alpha)$ and $\beta^-_\ep=\Pi^-_{L}\circ r_\ep (\alpha)$ where $r$ denotes restriction.  We have  the following lemma:

\begin{lemma}
$D_{\epsilon}$ is an isomorphism with inverse $P_\epsilon\oplus Q_\epsilon$ where $$P_\epsilon\oplus Q_\epsilon:
L^2([0,\ep]\times S^1,\C^n)\oplus L^2_{1/2}(S^1,\C^n)\rightarrow L^2_1([0,\ep]\times S^1,\C^n)$$ There exists $C>0$ such that
$||P_\epsilon||\leq  C $ and $||Q_{\epsilon}||\leq C$, independent of $\epsilon$. \\ We have $$||P_\epsilon(a)_{|\partial(
[0,\ep]\times S^1)}||_{L^2_{1/2}}\leq C||a||_{L^2}$$ and  $$||Q_\ep(b)_{|\partial([0,\ep]\times S^1)}-b||^2_{L^2_{1/2}}\leq 2\int_{[0,\ep]\times S^1 }|\partial_t
Q_\ep(b)|^2$$
\end{lemma}

\begin{proof}
$P_\ep$:  Take $\{\phi_\lambda\}$ an orthonormal eigenbasis for $L^2(S^1,\C^n)$.  We may write any element in
$L^2([0,\ep]\times (S^1,\C^n)$ as a sum $\sum_{\lambda}g_\lambda(t) \phi_\lambda$ with $g_\lambda \in L^2$.
 For $\lambda>0$ define $$P_\epsilon(g_\lambda \phi_\lambda)=\phi_\lam \int_0^te^{-\lambda(t-\tau)}g_\lambda(\tau)d\tau
 =h(t)\phi_\lam $$

 We compute the $L^2_1$ norm of $h(t)\phi_\lambda$.  Since $\lambda>0$ is bounded below (independently of $\ep$) we can
 use $\lambda^2\int_0^\ep |h|^2  + \int_0^\ep |\partial_t h|^2$ to compute the square of the $L^2_1$ norm of $h(t)\phi_\lambda$.  We have
 $$\int_{[0,\ep]\times Y}|D_\ep(h\phi_\lambda)|^2=\lambda^2\int_{0}^\ep |h|^2  + \int_0^\ep |\partial_t h|^2+\lambda |h(\ep)|^2$$
 This bounds the $L^2_1$ norm of $h(t)\phi_\lambda$ as well as the $L^2_{1/2}$ norm of $h(\ep)\phi_\lambda$ in terms of
 $||g\phi_\lambda||_{L^2}$. \\

$Q_\epsilon$: We have an explicit formula for the inverse:

$$Q_\epsilon(\phi_\lambda)=-e^{- t\lambda}\phi_\lambda, \lambda>0$$
$$Q_\epsilon(\phi_\lambda)=e^{-(t-\ep)\lambda}\phi_\lambda, \lambda\leq 0$$

\noindent  Lets consider the case $\lambda>0$.  We have $D_\epsilon \circ Q_\epsilon=0$.  The $L^2$-norm of
$Q_\epsilon(\phi_\lambda)$ is bounded by 1 since
  $e^{- t\lambda}\leq 1$.  We have $\partial_t(Q_\epsilon(\phi_\lambda))=-\lambda Q_\epsilon(\phi_\lambda)$ thus,
  $$\int_{[0,\ep]\times S^1}|\partial_t(Q_\epsilon(\phi_\lambda))|^2=\frac{\lambda(1-e^{-2\lambda\epsilon })}{2}$$

  Since $\partial_t\circ Q_\epsilon=L\circ Q_\epsilon$ and the $L^2_{1/2}$-norm of $\phi_\lambda$ is $|\lambda|^{1/2}$
  this establishes the desired bound. Finally, note that $$||\phi_\lambda-Q_\ep(\phi_\lambda)_{|\{\ep\}\times
  S^1}||^2_{L^2_{1/2}}=\lambda(1-
  e^{-\ep\lambda})^2 \leq \lam (1-e^{-2\ep \lam})= 2\int |\partial_t(Q_\ep(\phi_\lambda))|^2$$
  \end{proof}

In applications, we will have a nonlinear term for which we use the following lemmas:

\begin{lemma}
Given $f \in L^2_1([0,\epsilon]\times S^1,\C^n)$, vanishing on one of the ends, we have $||f||_{L^4} \leq
C||f||_{L^2_1}$ with $C$ independent of $\epsilon$.

\end{lemma}

\begin{proof}
The argument reduces to that of a function on $R^2$ with support in the rectangle $[0,1]\times [0,\ep]$.
We assume $f$ vanishes on say $\{0\}\times [0,\ep]$ and $[0,1]\times \{0\}$.  We have $$|f(x,y)|\leq \int |\partial_1f(x',y)|dx'\leq \ep^{1/2}(\int|\partial_1f(x',y)|^2dx')^{1/2} $$ Similar estimate with $\partial_2$ implies
$$|f(x,y)|^4\leq \ep \int|\partial_1f(x',y)|^2dx' \int|\partial_2f(x,y')|^2dy' $$ Integrating, gives
$$\int|f(x,y)|^4dxdy \leq \ep \int|\partial_1f(x,y)|^2dxdy \int|\partial_2f(x,y)|^2dydx \leq \ep(\int |\nabla f(x,y) |^2 dxdy)^2  $$

\end{proof}

\begin{lemma} \label{l4lem}
Given $\beta \in L^2_{1/2}(Y,E)$ and $v\in L^2([0,\ep]\times Y,E)$ we have $$||Q_\ep(\beta)+P_\ep(v)||_{L^4}\leq
C||\beta ||_{L^2_{1/2}}+C||v||_{L^2}$$

\end{lemma}

\begin{proof}  Decompose $v=v_++v_-$ into positive (nonpositive) eigenvectors of $L$.  We have $$||P_{\ep}(v)||_{L^4}\leq ||P_\ep(v_-)||_{L^4}+||P_\ep(v_+)||_{L^4}$$ By construction, each of these two terms vanishes on an end of the cylinder.  Thus, since $||P_\ep(v_-)||_{L^2_1} \leq C ||v||_{L^2}$ and $||P_\ep(v_+)||_{L^2_1} \leq C ||v||_{L^2}$ the conclusion holds for the $v$ term.

For $Q_\ep(\beta)$ we need to investigate terms of 3 types. Let $\beta=\beta_-+\beta_++\beta_0$ where the decomposition corresponds to breaking up $b$ into the positive, negative and zero eigenspaces.   For $\beta_0$,  we have $Q_\ep(\beta_0)=\beta_0$
Since there are only finitely many eigenvectors of $L$ with zero eigenvalue, we can bound the $L^4$-norm of $Q_\ep(\beta_0)$ by $C||\beta_0||_{L^2_{1/2}}$.  For $\beta_+$ we note that
although  $Q_\ep(\beta_+)=-\sum_\lambda e^{-t\lam} c_\lam \phi_\lam$ does not vanish on the endpoints it extends to an $L^2_1$
function on $[0,\infty]\times S^1$.  In fact, by the calculation in the first lemma, the $L^2_1$-norm of the extension
is bounded by the $L^2_{1/2}$ norm of $\beta_+$.  The lemma above applies since this extension vanishes at $\infty$.  The argument for $\beta_-$ is
similar.
\end{proof}
\begin{lemma}
$||Q_\ep||_{(L^2_{1/2}:L^4)}$ is  uniformly bounded in $\epsilon$ and approaches 0 weakly as $\ep \rightarrow 0 $.

\end{lemma}

\begin{proof}  Choose any $\delta>0$. We have $\beta=\sum_{i=1}^k c_i\phi_{\lam_i}+\beta'$ where $||\beta'||_{L^2_{1/2}}<\delta/C$.  As $\ep \rightarrow 0$, we have $||Q_{\ep}(\phi_{\lam_i})||_{L^4}\rightarrow 0$ since the $C^0$ norm of $Q_{\ep}(\phi_{\lam_i})$ is bounded by that of $\phi_{\lam_i}$ the length of the cylinder is going to zero.  Thus $||Q_\ep(\beta)||_{L^4}\leq \sum_{i=1}^k||c_iQ_\ep(\phi_{\lam_i})||_{L^4}+\delta$ and we can choose $\epsilon$ so small that $\sum_{i=1}^k||c_i Q_\ep(\phi_{\lam_i})||_{L^4}\leq \delta$
\end{proof}
\begin{lemma}
 We may write $X_{H}(x)$ as $K(x)\cdot x$ where $K$ is a function with $K(0)=0$ and $|K(x)|\leq C|x|$.
  Furthermore,  $|K(x)\cdot x-K(y)\cdot y|\leq 2C(|x|+|y|)|x-y|$.
\end{lemma}
\begin{proof}
Let $G(x)=X_{H}(x)$.  Since $H(x)=0$ when $x$ is near 0, we have $$G(x)=\int_0^1DG(tx)dt\cdot x$$
Let $K(x)=\int_0^1DG(tx)dt$.  Since $DG(0)=0$ and $|DG(tx)|\leq C|t||x|$, we have
 $K(0)=0$ and $|K(x)|\leq C|x|$ as desired.  Note that $$|K(x)-K(y)|\leq C|x-y|$$ since $$|DG(tx)-DG(ty)|\leq C|tx-ty|$$ We estimate: $$|K(x)\cdot x-K(y)\cdot y|\leq |K(x)\cdot x-K(x)\cdot y| +
 |K(x)\cdot y-K(y)\cdot y|\leq 2C (|x|+|y|)|x-y|$$
\end{proof}
\begin{lemma}
Given functions $\alpha$, $\beta$ on the cylinder $S^1\times [0,\ep]$, we have $$||X_{H}(\alpha)-X_{H}(\beta)||_{L^2}\leq 2C(||\alpha||_{L^4}+||\beta||_{L^4})||\alpha-\beta||_{L^4}$$ where $C$ is independent of $\ep$.
\end{lemma}
\begin{proof}
Integrating the inequality of the previous lemma we get:
 $$\int |X_{H}(\alpha)-X_{H}(\beta)|^2d\tet dt\leq$$ $$C^2\int(|\alpha|^2+|\beta|^2)|\alpha-\beta|^2d\tet dt\leq 4C^2(||\alpha||_{L^4}^{2}+||\beta||_{L^4}^2)||\alpha-\beta||_{L^4}^{2}$$
\end{proof}

Borrowing notation from the section on the shrinking cylinder, we deduce that $$v \mapsto X_{H}(Q_\ep(\beta)+P_\ep(v))$$ is a contraction mapping for small enough $\ep$ and a fixed $\beta$.   Let us denote this map by $F^\ep(\beta,v)$. 
Slightly more generally, we may consider $$F^\ep(\beta,v)=g_\ep-X_{H}(Q_\ep(\beta)+P_\ep(v))$$ where $g_\ep \in L^{2}([0,\ep]\times S^1, \C^n)$.  $F^\ep(\beta,)$ is contraction mapping with the same constants.  To ensure $F^\ep(\beta,)$ maps a ball of radius $1/8C$ to itself we must suppose that $g$ is sufficiently small.\\\\
Consider the map $G^\ep(\beta,v)=(\beta,v-F^\ep(\beta,v))$.  For $\ep$ small, the previous lemma allows us to conclude that the existence of an inverse $H^\ep(\beta,v)$ with $G^\ep(\beta,H^\ep(\beta,v))=(\beta,v)$.  Furthermore, $H^\ep(\beta,0)\rightarrow 0$ as $\ep \rightarrow 0$.  We would like to conclude the same for the derivative:
\begin{lemma}
Let $D_1H_{(\beta,0)}^\ep$ be the derivative with respect to the $\beta$ variable at the point $(\beta,0)$.  We have $|D_1H^\ep_{(\beta,0)}|\rightarrow 0$ as $\ep\rightarrow 0$.
\end{lemma}
\begin{proof}
 Pick $b\in T_{\beta}L^2_{1/2}([0,\ep]\times S^1,\C^n)$. $F^\ep(\beta,H^\ep(\beta,0))=H^\ep(\beta,0)$ implies $$D_1F^\ep_{(\beta,H^\ep(\beta,0))}( b)+D_2 F^\ep_{(\beta,H^\ep(\beta,0))}(D_1H^\ep_{(\beta,0)}(b))=D_1H^\ep_{(\beta,0)}(b)$$  The desired result will follow if we can estimate the LHS.  From the definition, $$D_1F^\ep_{\beta,v}( b)=\nabla K(Q_\ep(\beta))Q_\ep(b)\cdot(P_\ep(v)+Q_\ep(\beta))+K(P_\ep(v)+Q_\ep(\beta))\cdot(Q_\ep(b))$$ and $$D_2F^\ep_{\beta,v}(w)=\nabla K(P_\ep(v))(P_\ep(w))\cdot(P_\ep(v)+Q_\ep(\beta))+K(P_\ep(v)+Q_\ep(\beta))\cdot(P_\ep(w))$$
Plugging in $v=H^\ep(\beta,0)$ and $w=D_1H^\ep(\beta,0)$ and using that $Q_\ep(\beta)$ and $P_\ep(H^\ep(\beta,0))\rightarrow 0$ as $\ep\rightarrow 0$ the conclusion follows.
\end{proof}

\subsection{Adding the Collar}

\begin{lemma}\label{collem}
Given $b\in \Be$,  decompose $b$ as $b^+_0+b^-_\ep$.  There exists  sufficiently small $\ep>0$ such that there is  a unique small energy holomorphic curve $\gamma$ with $b^+_0+b^-_\ep$ as the mixed boundary value.
\end{lemma}

 \begin{proof}  This follows immediately from the arguments of the previous subsections.  Indeed, let $P_\ep=P_\ep^0\oplus P_\ep^+$ as above.  If $v$ is the unique small fixed point of the map $$v\mapsto g-X_H(Q_\ep(b)+P_\ep(v))$$ then $$(\partial_t+J\partial_\theta)(P_\ep( v)+Q_\ep(b))+X_H(Q_\ep(b)+P_\ep(v))=g$$  Note that since $g$ is some fixed section with $L^2_1$ norm and $L^4$ norm going to zero as $\ep \rightarrow 0$ the contraction lemma applies.
 \end{proof}

 We complete $\cup_{t\in (0,1]}
\Mod_t$  to form a lc-manifold $\cup_{t\in [0,1]}
\Mod_t$ as follows.  The 0th stratum is  $\cup_{t\in (0,1)}\Mod_t$ while the 1st stratum is  $\Mod_1 \coprod \Delta$.  A sequence $z_i \in \cup_{t\in (0,1)}\Mod_t$ is said to converge to $z \in \Delta$ if $\Rest(z_i)$ converges in $L^2_{1/2}$ to $z$.  Lemma $\ref{collem}$ implies that the completed space has the structure of an lc-manifold.  In fact, the completed manifold is smooth, but the extension of the restriction map to the diagonal is not smooth.  We will see how to deal with this in a later section.

\subsection{Verifying the Axioms}

In this section we verify that the completed correspondence satisfies the first 3 axioms of a correspondence. This ensures that for a chain $\sigma$, we will have that $\sigma \times_{\Be_0}\cup_{t\in [0,1]}
\Mod_t$ satisfies the first 3 axioms of a chain.  In other words, those axioms that deal with the convergence properties.  The existence of a lc-structure on $\sigma \times_{\Be_0}\cup_{t\in [0,1]}
\Mod_t$ will be handled in a separate section.  Note, that it is possible to modify the definition of a correspondence so that $\cup_{t\in [0,1]}
\Mod_t$ is a genuine correspondence and thus $\sigma \times_{\Be_0}\cup_{t\in [0,1]}
\Mod_t$ automatically has such an lc-structure.  We choose to avoid this more general definition since it seems to obscure matters and will not be used in the future.  \\\\
We will first deduce a uniform $L^2_1$ bound on configurations.  Suppose  for this section that $\Energy(\gamma_i)$ is bounded and $\Rest_0(\gamma_i)$ weakly converges (this hypothesis is satisfied in all of the axioms we need to check).  We may also assume that we have a sequence of solutions $\gamma_i$ on cylinders of shrinking length as all the other cases have been covered.  Energy bounds give us uniform bounds on $L^2_1$-norm of $\gamma_i$ by theorem $\ref{compact}$

\begin{lemma}
We have a uniform bound on $||\Rest_\ep(\gamma_i)||_{L^2_{1/2}}$. Assume, $\Rest_0(\gamma_i)$ is $L^2_{1/2}$ convergent and $\Energy(\gamma_i)\rightarrow 0$.  We have $\Rest_{\ep}(\gamma_i)$ $L^2_{1/2}$-convergent and $\gamma_i$ $L^2_1$-convergent.
\end{lemma}

\begin{proof}
 For strong compactness, we assume $\Energy(\gamma_i)\rightarrow 0$ and $\Rest_0(\gamma_i)$ converges.  By the  arguments on weak compactness, convergent of $\Rest_0(\gamma_i)$ implies the same of $\Rest_{\ep}(\gamma_i)$. Note, since the energy is approaching zero, eventually the sequence is in the domain of the contraction mapping theorem and thus lies in the collar.  Therefore, the endpoints uniquely parameterize the solutions $\gamma_i$ and strong convergence follows from that of the endpoints.
\end{proof}

Observe that our discussion establishes  the following:
\begin{corollary}
 If $\Energy(\gamma_i)<C$ and $\Rest_0(\gamma_i)$ is uniformly bounded, we have $||\gamma_i||_{L^2_1}$ is uniformly bounded
\end{corollary}
We are in good shape to verify the axioms: \\
Axiom $1'$:  $\CSD_0-\CSD_1$ is bounded below by 0 since energy $\CSD$ is nonincreasing on trajectories.  To establish lower semi-continuity, we claim that in fact any sequence $(\Rest_0(\gamma_i),\Rest_\ep(\gamma_i))$ as above  weakly converges (up to a subsequence) to a diagonal element. We can assume  $(\Rest_0(\gamma_i),\Rest_\ep(\gamma_i))$ is strongly $L^2$ convergent.  We have: $$||\Rest_0(\gamma_i)-\Rest_\ep(\gamma_i)||_{L^2}\leq \int_0^\ep||d\gamma_i/dt||_{L^2}\leq \ep^{1/2}||\gamma_i||_{L^2_1}\leq C\ep^{1/2}$$ since $||\gamma_i||_{L^2_1}$ is uniformly bounded.  This implies the claim and thus the lower semi-continuity of $\CSD_1-\CSD_0$.  Axiom $2'$ and $3'$ have been verified in the previous lemma.

\subsection{Concluding the Proof}
Given a cycle $\sigma:P\rightarrow \Be$ we verify that $P\times_{\Be}\cup_{t\in [0,1]}\Mod_t$
is a cobordism between $P$ and  $P\times_{\Be}\Mod_1$.  Given that $P$ is transverse to $\cup_{t\in (0,1]}
\Mod_t$, the axioms of a cycle for $P \times_{\Be} \cup_{t\in [0,1]}
\Mod_t$ were verified in the preceding sections.  What needs to be checked is that it has the structure of a lc-manifold.  The potential problem occurs near the diagonal where the
 restriction maps converge in $C^0$ to the inclusion map. More precisely,
 in the collar coordinates $(b^+,b^-,t)$, $\Rest_0(b^+,b^-,t)=b^++e^{tL_1}b^-+G_0(t,b)$
 where $G_0(t,b)$ as a function of $b$ converges to 0 in $C^1$ topology as $t\rightarrow 0$.
 Note that $e^{tL_1}_{|B^-}$ is a family of compact operators converging to the identity in the weak topology.  Thus, we don't have $C^1$ convergence for the restriction map.  \\

We will work locally, so assume that $P$ is a ball around
the origin in Hilbert space.  Let $b_0=\sigma(0)$. Near $b_0$, $\Be$ is an affine space modeled on $T^-\Be\oplus T^+\Be$.
By assumption, $\pi^-\circ D\sigma$ is Fredholm.  Applying the inverse function theorem to $\pi^-\circ \sigma $,
we can find coordinates for $\sigma$ so that $\sigma(p)=b_0+f(p)+A(p)$ where $Df_p$ is compact at all $p$
and $A$ is an linear Fredholm map $$A:TP\rightarrow T^-\Be$$

 In these  coordinates, the map $\Rest_0-\sigma: P\times  \cup_{t\in [0,1]}\Mod_t\rightarrow \Be$ can written  as $$(p,b^+,b^-,t)\rightarrow (e^{tL}-1)b^-_0+b^+-A(p)+e^{tL}b^--f(p)+G_0(t,b_0+b)$$
 Pick a left inverse $A^{-1}$ for $A$.  Thus, $A\circ A^{-1}-I$ has finite rank on $T^-\Be$.  Define a change of coordinates by $$(b^+,b^-,p,t)\mapsto (b^+,b^-,\tilde{p},t)$$ with $\tilde{p}=p-A^{-1}\circ e^{tL_1}b$.  This is a homeomorphism with inverse taking $\tilde{p}$ to $p=\tilde{p}+A^{-1}e^{tL}b^-$. Note that for each fixed $t\geq 0$, the map is a diffeomorphism.  With the new coordinates the map $\Rest_0-\sigma$ becomes $$(\tilde{p},b^+,b^-,t)\rightarrow (e^{tL}-1)b^-_0+b^+-A(\tilde{p}+A^{-1}e^{tL}b^-)+e^{tL}b^--f(\tilde{p}+A^{-1}e^{tL}b^-)+G_0(t,b_0+b)$$ This can be simplified to $$(\tilde{p},b^+,b^-,t)\rightarrow (e^{tL}-1)b^-_0+b^+-A(\tilde{p})+K\circ e^{tL}b^--f(\tilde{p}+A^{-1}e^{tL}b^-)+G_0(t,b_0+b)$$
 where $K=I-A\circ A^{-1}$ is a finite rank operator.
   Since the change of coordinates is a homeomorphism, the continuity of the map up to the boundary still holds.  We claim that the differential in the $\tilde{p}$ and $b$ variables converge as $t\rightarrow 0$ to the ones for $t=0$.
   Computing the differential at $(\tilde{p},b,t)$ we have:
   $$(\delta \tilde{p},\delta b^+,\delta b^-)\mapsto \delta b^+-A(\delta \tilde{p})+K\circ e^{tL}\delta b^--Df_{p}(\delta \tilde{p}+A^{-1}e^{tL}\delta b^-)+{DG_0}_{(t,b_0+b)}(\delta b^++\delta b^-)$$
   We want this differential to converge to $$(\delta \tilde{p},\delta b^+,\delta b^-)\mapsto \delta b^++K\delta b^-+Df_p(\delta \tilde{p}+A^{-1}\delta b^-)$$
   In view of the compactness of $Df_p$ and $K$ as well as the fact that $e^{tL}$ is self-adjoint, the claim is a consequence of the following lemma proved in the appendix:
\begin{lemma}\label{weakconv2}
   Given a uniformly bounded weakly converging sequence of operators $A_i:V\rightarrow W$ between Hilbert spaces and  a strongly convergent sequence of compact operators $K_i:W\rightarrow U$, $K_i\circ A_i$ converge strongly to $K_\infty\circ A_\infty$ provided $A_i^*$ converge weakly to $(A_\infty)^*$.
  \end{lemma}

  Thus, we have found coordinates where the difference $\Rest_0-\sigma$ is lc-smooth so the inverse function theorem with parameter implies $(\Rest_0-\sigma)^{-1}(0)$ is a  manifold with boundary.   Finally, we need to verify that the projection to the other end is smooth in the new coordinates.  This time the
   map is $$(p,b^+,b^-,t)\mapsto e^{-tL_1}(b^++b^+_0)+b^-+b^-_0+G_1(t,b)$$ where again $G_1(t,b)$ as a function of $b$ converges to 0 in $C^1$ topology as $t\rightarrow 0$. Notice, however, restricted to the fiber product $b^+=\pi^+(\sigma(p))$ and thus, restricted to the fibre product, the map may be written using the $\tilde{p}$ coordinates as $$(\tilde{p},b^+,b^-,t)\mapsto e^{-tL_1}(\pi^+(\sigma(\tilde{p}+A^{-1}e^{tL}b^-))+b^+_0)+b^-+b^-_0+G_1(t,b)$$  The derivative is:
   $$(\delta \tilde{p},\delta b^+,\delta b^-)\mapsto e^{-tL_1}(\pi^+\circ D\sigma_p(\delta \tilde{p}+A^{-1}e^{tL}\delta b^-))+\delta b^-+{DG_1}_{(t,b)}(\delta b)$$
    Since  $\pi^+\circ D\sigma_p$ is compact for any $p$,  we have $e^{-tL}\circ \pi^+\circ D\sigma_p$ compact as well.  The following lemma (also proved in the appendix) implies $e^{-tL}\circ \pi^+\circ D\sigma_p$ is converging:
\begin{lemma}\label{weakconv2}
   Given a uniformly bounded weakly converging sequence of operators $A_i:V\rightarrow W$ between Hilbert spaces and  a strongly convergent sequence of compact operators $K_i:U\rightarrow V$, $A_i\circ K_i$ converge strongly to $A_\infty\circ K_\infty$.
  \end{lemma}
From the lemma we conclude that $e^{-tL}\circ \pi^+\circ D\sigma_p$ is a convergent sequence of compact operators, hence we can apply the previous lemma to conclude that $e^{-tL}\circ \pi^+\circ D\sigma_p\circ A^{-1}e^{tL}$ is converging as well.  This completes the proof of the existence of an lc-structure on $P\times_{\Be}\cup_{t\in [0,1]}\Mod_t$

\section{The Existence of a Critical Point}
\label{Appl}
One of the central themes in \cite{Hof} is how the existence of a critical point of the function $\CSD_H$ leads to a variety of applications in symplectic geometry.  For example, the celebrated nonsqueezing theorem is a rather direct consequence of the existence of a symplectic capacity which in turn is defined crucially using the existence  of critical points of $\CSD_H$.  In \cite{Hof} this is established by using compactness properties of the regularized gradient flow of $\CSD_H$ as well as the Leray-Schauder degree theory.  In this section we will demonstrate how the existence of a critical point can be established using the unregularized gradient by appealing to the theory developed in this work.
Recall that we have the decomposition $T\Be=T^+\Be\oplus T^-\Be$. We may identify $T\Be$ with $\Be$.  Fix a unit vector $e^+\in T^+\Be$.  We assume $H$ is smooth with $H=0$ near $0$ and $H(x)=(1+\ep)|x|^2$ for $|x|$ large.  Recall the definition of $\Gamma_\alpha$ and $\Sigma_\tau$ from section $\ref{CycleDef}$. \\\\
In our semi-infinite setting there is a notion of gradient flow on cycles. Indeed, given a cycle $\sigma: P\rightarrow \Be$, let  $$\GF_t(\sigma)=\sigma\times_{\Be} \Mod_t$$  The shrinking argument applies in this case to show that the new cycle is cobordant to the original one. 

\begin{theorem}
There exists a critical point $x$ of $\CSD_H$ with $\CSD_H(x)\geq \beta$.
\end{theorem}
\begin{proof}
 We argue by contradiction.  Assume no such critical point exists.  Then $\CSD_H(\GF_t(\Gamma_\alpha))>C$ for any $C>0$
 given that $t$ is sufficiently large.  This is a familiar argument from Morse theory. By contradiction, if there exists a sequence of increasingly  long gradient flow lines that start on $\Gamma_\alpha$ and have bounded energy.  There then would be a subsequence of trajectories of some fixed length with energy converging to 0.  By compactness, such a subsequence would converge to a critical point $x$ with $\CSD(x)>0$.  This is a contradiction.  On the other hand, we then would
  have $$\GF_t(\Gamma_\alpha)\times_{\Be}\Sigma_\tau=\emptyset$$ since $\CSD_H$ restricted to $\Sigma_\tau$ is
   bounded above.  This is impossible since $\GF_t(\Gamma_\alpha)$ is cobordant to $\Gamma_\alpha$ by a cobordism
   staying away from points where $\CSD_H\leq 0$ and thus from $\partial(\Sigma_\tau)$, while $\Gamma_\alpha$ intersects $\Sigma_\tau$ transversely in a point.
\end{proof}

\section{Appendix: Weakly Convergent Operators }
\begin{definition}
A sequence of operators $A_i:V\rightarrow W$ between Hilbert spaces is said to converge weakly, if there exists a bounded operator $A_\infty:V\rightarrow W$ such that for any $v\in V$ we have $A_i(v)\rightarrow A_\infty(v)$.
\end{definition}

\begin{lemma}\label{weakconv1}
   Given a uniformly bounded weakly converging sequence of operators $A_i:W\rightarrow U$ between Hilbert spaces and a compact operator $K:V\rightarrow W$, $A_i\circ K$ converge strongly to $A_\infty\circ K$.
  \end{lemma}
  \begin{proof} Taking the new sequence $A_i-A_\infty$ we can assume $A_\infty=0$.  By contradiction, suppose  there exists a sequence $v_i$ with $|v_i|=1$ and $| A_i\circ K(v_i)|\geq C>0$.  Since $K$ is compact, the elements $w_i=K(v_i)$ have a converging subsequence $w_j$ with limit $w_\infty$.  By assumption, $\lim_i(A_i(w_\infty))=0$.  Since $A_i$ are uniformly bounded we have $$A_j(w_j)=A_j(w_j-w_\infty)+A_j(w_\infty)$$  Since $A_j$ are uniformly bounded and we have $$\lim_j|A_j(w_j-w_\infty)|\leq const\cdot \lim_j|w_j-w_\infty|=0$$ and thus $\lim_jA_j(w_j)=0$ contradicting the fact that $$|A_j(w_j)|=|A_j\circ K(v_j)|\geq C>0$$
  \end{proof}

Similarly we have:
\begin{lemma}\label{weakconv2}
   Given a uniformly bounded weakly converging sequence of operators $A_i:V\rightarrow W$ between Hilbert spaces and a compact operator $K:W\rightarrow U$.  Assume $A_i^*$ is also weakly converging with limit $(A_\infty)^*$. We have, that $K\circ A_i$ converge strongly to $K\circ A_\infty$.
  \end{lemma}
  \begin{proof} Apply the previous first lemma to $A_i^*$ and $K^*$.
   \end{proof}
Finally, combining the previous two lemmas we obtain:
\begin{lemma}\label{weakconv2}
   Given a uniformly bounded weakly converging sequence of operators $A_i:V\rightarrow W$ and $A_i':V'\rightarrow W'$  and a strongly convergent sequence of compact operators $K_i:W\rightarrow V'$, $A'_i\circ K_i \circ A_i$ converge strongly to $A'_\infty \circ K_\infty \circ A_\infty$.
  \end{lemma}
  \begin{proof}
  We have $$A'_i\circ K_i \circ A_i-A'_i\circ K_\infty \circ A_i=A'_i\circ (K_i-K_\infty )\circ A_i$$ and
  $$|A'_i\circ (K_i-K_\infty )\circ A_i|\leq |A_i'||K_i-K_\infty||A_i|$$  Thus, since $A_i$ and $A'_i$ are uniformly bounded, it suffices to assume $K_i=K_\infty$.  On the one hand, the previous lemmas imply that the uniformly bounded sequence of compact operators $T_i=K\circ A_i$ is strongly convergent.  Now, apply the same argument to $A'_i\circ T_i$.
   \end{proof}

\section{Appendix: lc-Manifolds }

\subsection{Spaces Stratified By Hilbert Manifolds}
\begin{definition}
A second countable Hausdorff space $P$ has a \term{stratification by Hilbert manifolds of depth $k$},  if $P^k\subset P^{k-1} \ldots P^0=P$ where for each $i$, $P^i$ is closed in $P$ and the open stratum $P^i-P^{i+1}$ is Hilbert manifold.  A \term{stratum smooth map} $f: P\rightarrow X$ where $X$ is a Hilbert manifold is a continuous map smooth on each open stratum.  Such a map is said to be \term{transverse} to a submanifold $Y \subset X$ if it is transverse on each stratum.
\end{definition}

Note that the product of $P$ and $Q$, for any two such spaces, is also stratified by Hilbert manifolds.

\subsection{Locally Cubical Hilbert Manifolds}

Let $\simp(k)=[0,1)^k$.  We view $\simp(k)$ is a stratified space in the natural way.  Let $\simp$ denote a typical coordinate in $\simp(k)$.  At times, by abuse of notation, we let $\simp(k)$ denote a neighborhood of the origin in $[0,1)^k$

\begin{definition}
Given a space $P$ stratified by Hilbert manifolds and an open set $V$ in $P^i-P^{i+1}$, a  \term{locally cubical Hilbert manifold} chart about $V$ is an embedding (of stratified Hilbert manifolds)
$$f:V\times \simp(i)\rightarrow U\subset P$$
where $U$ is open in $P$ and $f(v,0)=v$.
\end{definition}

\begin{definition}

A \term{locally cubical Hilbert manifold} (or lc-manifold for short) is a stratified Hilbert space $P$ with a cover by locally cubical charts as above with no further compatibility assumptions (other than those impossed by the being embeddings of stratified Hilbert manifolds).

\end{definition}

\begin{lemma}
The product of two lc-manifolds is a lc-manifold.
\end{lemma}
\begin{proof}
Since $\simp(k)\times\simp(k')=\simp(k+k')$ in a canonical way,  a cover is specified by charts of the form $V\times V' \times \simp (k+k')$.
\end{proof}

\begin{definition}
A \term{smooth map} $\sigma$ from a lc-manifold $P$ to a Hilbert manifold $X$ is a stratum smooth map such that each point has at least one chart $V\times \simp(i)$ where $\sigma$ has the form $\sigma(v,\simp)$ with  $\sigma$ is smooth in the $v$ coordinates and, along with its $v$-derivative, continuous in the $\simp$ coordinates.
\end{definition}

\noindent \textbf{Remark.} The reason for restricting to lc-maps as opposed to say smooth maps from manifolds with corners will become apparent when dealing the shrinking cylinder argument in the section $\ref{Triv}$.\\\\
\noindent \textbf{Remark.} Given a smooth map $\sigma:P \rightarrow X$ as in the previous lemma and a smooth map $f: X\rightarrow Y$ of Hilbert manifolds the composition $f \circ \sigma:P \rightarrow Y$ is also smooth.

\begin{lemma}
Given a smooth map $\sigma:P \rightarrow X$ as in the previous lemma and a closed submanifold $Y \subset X$ such that $\sigma$ is transverse to $Y$, $\sigma^{-1}(\Delta)$ is an lc-manifold.
\end{lemma}
\begin{proof}
Note that $\sigma^{-1}(\Delta)$ is naturally a space stratified by Hilbert manifolds.  In a chart, we are reduced to the following local situation.  Given $\sigma: V\times \simp(k) \rightarrow W$ where $V,W$ are Hilbert spaces and $f$ is smooth in the $v$ variables and, along with its first $v$-derivative, continuous in the $\simp$ variables.  Assume $\sigma_0(v)=\sigma(v,0)$ is has 0 as a regular value.  Then, locally there is a stratum preserving smooth homeomorphism  $\sigma^{-1}_{0}(0)\times \simp(k) \rightarrow \sigma^{-1}(0)$. This, in turn, follows from the inverse function theorem with dependence on a parameter.
\end{proof}

%\subsection{Boundary Operator}

%Note that a stratified space such as $V\times \simp(k)$ has a natural order on its open $k-1$ dimensional stratum since $(V\times \simp(k))^{k-1}=V\times \cup_i (\simp(k-1),f_i)$ where $f_i:\simp(k-1)\rightarrow \simp(k)$ are the face inclusion maps obtained by omitting the $i$th coordinate.

%\begin{definition}
%An lc-manifold $P$ is said to be \term{ordered} if we are given a maximal atlas of  lc-charts so that the overlap maps preserve the natural ordering on the strata.
%\end{definition}

%\begin{definition}
%Define $\partial\simp(k)=\coprod_i (\simp(k-1),f_i)$
%\end{definition}

%\begin{definition}
%Given a lc-manifold $P$ we construct $\partial P$ in steps.  As the top stratum take $P^1-P^2$. Given a chart of the type $V\times \simp(2)$ glue in $V\times \partial \simp(2)$ to $P^1-P^2$ using the inclusion maps.  In general, to obtain the $k$th stratum, glue in $V\times \partial \simp(k+1)$.
%\end{definition}

%\noindent \textbf{Remark.} $\partial P$ has a natural lc-manifold structure and inherits an ordering if $P$ is ordered.

%\begin{lemma}
%If $P$ is ordered, then $\partial^2 P$ is naturally a disjoint union of two identical lc-manifolds $P_1$ and $P_2$.
%\end{lemma}
%\begin{proof} Locally, $\partial^2 (V\times \simp(k))=V\times \partial^2\simp(k)$.  Now, $\partial^2\simp(k)=\coprod_{i\neq j}(\simp(k-2),\sigma_{ij})$  where $\sigma_{ij}$ is the inclusion of the $(k-2)$ face obtained by ignoring the $i^{th}$ and $j^{th}$ factor.  Thus $\partial^2\simp(k)$ is naturally a disjoint union of two spaces corresponding to when $i<j$ and $i>j$. \end{proof}

\end{document}